\documentclass[10pt,a4paper,reqno]{amsart}
\usepackage[utf8]{inputenc}
\usepackage{amsmath}
\usepackage{amsthm}
\usepackage{amsfonts}
\usepackage{amssymb}
\usepackage{graphicx}
\usepackage{lmodern}
\usepackage[left=3.5cm,right=3.5cm,top=3.5cm,bottom=3.5cm]{geometry}

\usepackage{hyperref}

\usepackage{etoolbox}							
\patchcmd{\section}{\scshape}{\bfseries}{}{}
\makeatletter
\renewcommand{\@secnumfont}{\bfseries}
\makeatother

\newtheorem{theorem}{Theorem}					
\newtheorem{lemma}{Lemma}[theorem]

\numberwithin{equation}{section}

\author{Fabian Schneider}
\title[How likely are two recurrent events to overlap during a given time?]{\small How likely are two independent recurrent events to occur simultaneously during a given time? \normalsize}
\address{Hamburg, Germany}
\email{\url{fabian.schneider@desy.de}, \url{research@fschneider.info}}

\urladdr{\url{http://www.fschneider.info/}}
\date{December 22, 2016}

\begin{document}

\begin{abstract}
We determine the probability $P$ of two independent events $A$ and $B$, which occur randomly $n_A$ and $n_B$ times during a total time $T$ and last for $t_A$ and $t_B$, to occur simultaneously at some point during $T$. Therefore we first prove the precise equation 
\begin{equation*}
P^* = \dfrac{t_A+t_B}{T} - \dfrac{t_A^2+t_B^2}{2T^2}
\end{equation*}
for the case $n_A = n_B = 1$ and continue to establish a simple approximation equation
\begin{equation*}
P \approx 1 -  \left( 1 - n_A \dfrac{t_A + t_B}{T} \right)^{n_B}
\end{equation*}
for any given value of $n_A$ and $n_B$. Finally we prove the more complex universal equation
\begin{equation*}
P = 1 - \dfrac{ \left( T^+ - t_A n_A - t_B n_B \right)^{n_A + n_B} }{ \left( T^+ - t_A n_A  \right)^{n_A} \left( T^+ - t_B n_B \right)^{n_B} } \pm E^\pm,
\end{equation*}
which yields the probability for $A$ and $B$ to overlap at some point for any given parameter, with $T^+ := T + \frac{t_A + t_B}{2}$ and a small error term $E^\pm$.
\end{abstract}

\maketitle

\section{Introduction}
\label{sec:intro}

Let us consider two independent and recurring events $A$ and $B$, which take place during a total time $T$. The events occur exactly $n_A$ and $n_B$ times randomly during this total time and last for $t_A$ and $t_B$ until the next event occurrence may happen. Since both events occur independently, they may eventually take place simultaneously at some point during $T$ -- i.e. they overlap. This may happen in case an occurrence
\begin{enumerate}
\item of $A$ and $B$ take place at the same moment,
\item of $A$ takes places while an occurrence of $B$ still takes place or
\item of $B$ takes place while an occurrence of $A$ still takes place.
\end{enumerate}
We divide all time parameters in discrete units $\Delta \longrightarrow 0$ and let $P$ denote the probability for at least one overlap -- i.e. it exists at least one time unit, during which both events take place. We will derive expressions to determine $P$ under given parameters $T$, $t_A$, $t_B$, $n_A$ and $n_B$. The following examples and notes clarify the setting and will be revisited later: \\ \\
\emph{Example I} --- John works inside his office for 2 hours. A blue car will occur 10 times on the nearby street and remains visible for 1 minute each time. John looks outside 5 times for 3 minutes each. How likely is John going to see a blue car?: $T = 120 ~min$, $t_A = 3 ~min$, $t_B = 1 ~min$, $n_A = 5$,  $n_B = 10$ and we are about to determine $P(120,3,1,5,10)$. \\ \\
\emph{Example II} --- We choose a total time of 3 seconds. Both events happen only once for 1 second during this time: How likely are both events to overlap at some point?: $T = 3 ~sec$, $t_A = 1 ~sec$, $t_B = 1 ~sec$, $n_A = 1$,  $n_B = 1$ and we are about to determine $P(3,1,1,1,1)$.\footnote{Note: Since all time values are given with same unit, we will stop to note the unit in further calculations.} \\

\newpage
\noindent Definitions and Notes:
\begin{enumerate}
\renewcommand{\theenumi}{\alph{enumi}}
\item \textit{If the duration time of $A$ and $B$ differs, let always $A$ denote the event with the longer duration time: $t_A  \geq t_B$. \\
\item The duration of the event-overlap is not relevant -- i.e. an overlap of a few minutes or a millisecond are both counted as valid overlap. \\
\item We consider a possible overlap with the end of the total time as more natural and thus valid - e.g. John may start to look outside 4 seconds before the 120 minutes total time are over. \\
\item Both events occur precisely $n_A$ and $n_B$ times - we do not perceive these parameters as random in the proofs. \emph{When} the events occur is the only random aspect first. We remark in Section \ref{sec:nn} how the equation changes with $n_A$ and $n_B$ as random. }\\
\item \textit{Occurrences of the same event must not overlap with themselves -- e.g. John must not start to look outside while he is already looking. Thus $T \leq t_A \cdot n_A$ implies $P = 1$. }
\end{enumerate}

\section{Notation}\label{sec:not}

\noindent Overview of main notation: \\

\begin{tabular}{rcl}
$A$, $B$ & -- & labels of the events \\
$T$ & -- & total time \\
$t_A$, $t_B$ & -- & duration of event $A$, $B$ \\
$n_A$, $n_B$ & -- & number of occurrences of event $A$, $B$ during total time \\
$P$ , $\bar{P}$ & -- & probability for an (no) overlap \\
$\Delta$ & -- & infinitesimal unit, divides time in discrete pieces \\
&
\end{tabular}

\noindent Overview of specific notations for local purpose: \\

\begin{tabular}{rcl}
$S_1$, $S_2$ & -- & $\sum_{i=0}^{\frac{t_B}{\Delta} - 2} \left( \frac{t_A}{\Delta} + i \right)$ \ ,\ \ $\sum_{j=1}^{\frac{t_A}{\Delta} - 1} \left( \frac{t_A}{\Delta} + \frac{t_B}{\Delta} - 1 - j\right)$ \\ 
$T'$ & -- & $T / \Delta$ \\
$N_A$ , $N_B$ & -- & $\left(\frac{t_A}{\Delta} - 1\right) \cdot n_A$ \ ,\ \ $\left(\frac{t_B}{\Delta} - 1\right) \cdot n_B$ \\
$M_A$, $M_B$, $M$ & -- & $T' - N_A - n_A$ \ ,\ \ $T' - N_B - n_B$ \ ,\ \ $T' - N_A - n_A - N_B - n_B$ \\
$T^+$ & -- & $T + \frac{t_A + t_B}{2}$ \\
$\alpha_A $, $\alpha_B $ & -- & $ T + t_A - t_A n_A - t_B n_B$ \ ,\ \ $ \alpha_A - \frac{t_A-t_B}{2}$ \\
$\tau$ & -- & $\max\{ t_An_A , t_Bn_B \}$
\end{tabular}

\section{Precise Equation for $n_A = n_B = 1$}\label{sec:preciseEqn}

\noindent We start with deriving an equation for the case $n_A = n_B = 1$, since
\begin{enumerate}
\item this case has shown to occur unexpectedly often in applications 
\item prepares main techniques for the proof of Theorem \ref{theo:uniq} and
\item the equation yields precise results without error term\footnote{although $\vert t_A-t_B \vert > 0$ (see Section \ref{sec:preciseEqn})}.
\end{enumerate}
Let $P^*(T,t_A,t_B)$ denote the probability function for this case.

\begin{theorem}
The probability\footnote{Note: Due to reading purpose we will use just '$A$' instead of 'the $A$ occurrence(s)' -- same for $B$.} of $A$ and $B$ to overlap at some point during the total time $T$ in case of $n_A=n_B=1$ is given by
\begin{equation*}
P^*(T,t_A,t_B) = \dfrac{t_A+t_B}{T} - \dfrac{t_A^2+t_B^2}{2T^2}.
\end{equation*} 
\end{theorem}

\begin{lemma}\label{lem:gauss}
We remember the well known arithmetic summation formula
\begin{equation*}
1 + 2 + ... + n = \sum\limits_{k=1}^n k = \frac{1}{2}n \cdot (n+1).
\end{equation*}
\end{lemma}

\begin{proof}
See \cite{SUM}.
\end{proof}

\begin{proof}[Proof of Theorem.]
First, we divide the time in pieces of size $\Delta$. Thus, we get $T/\Delta$ as limited number of positions at which the events may start, so there are $(T/\Delta)^2$ possible arrangements. By determining the number of positions $x_B(x_A)$ of $B$ to overlap with $A$ for every possible position $x_A$ of $A$, we are able to obtain the probability $P^*$ with
\begin{equation}\label{eq:p1sum}
P^* = \lim_{\Delta \longrightarrow 0} \dfrac{\sum_{x_A} x_B(x_A)}{(T/\Delta)^2}.
\end{equation}
In general, if $A$ is located somewhere in the center of $T$ -- i.e. in a sufficient distance to beginning and end of the total time -- there are
\begin{equation}\label{eq:possPositions}
\frac{t_A}{\Delta} + \frac{t_B}{\Delta} - 1
\end{equation}
possible positions $x_B$ for $B$ to overlap with $A$. This holds true for 
\begin{equation*}
\frac{T}{\Delta} - \frac{t_A}{\Delta} - \frac{t_B}{\Delta} + 1
\end{equation*}
positions $x_A$ of $A$, since there are at $\frac{t_B}{\Delta}-1$ positions (beginning of $T$) and at $\frac{t_A}{\Delta}-1$ positions (end of $T$) less possible positions for $B$ to overlap with $A$. Therefore (\ref{eq:p1sum}) becomes
\begin{align} 
\nonumber P^* &= \lim_{\Delta \longrightarrow 0} \dfrac{ \sum\limits_{i=0}^{\frac{t_B}{\Delta} - 2} \left( \frac{t_A}{\Delta} + i \right) + \left( \frac{t_A}{\Delta} + \frac{t_B}{\Delta} - 1 \right) \left( \frac{T}{\Delta} - \frac{t_A}{\Delta} - \frac{t_B}{\Delta} + 1 \right) + \sum\limits_{j=1}^{\frac{t_A}{\Delta} - 1} \left( \frac{t_A}{\Delta} + \frac{t_B}{\Delta} - 1 - j\right) }{(T/\Delta)^2} \\[1em]
 &= \lim_{\Delta \longrightarrow 0} \dfrac{ \Delta^2 S_1 + \left( t_A + t_B - \Delta \right) \left( T - t_A - t_B + \Delta \right) + \Delta^2 S_2 }{T^2}. \label{eq:frac}
\end{align}
By applying Lemma \ref{lem:gauss} the first sum simplifies to
\begin{align*}
\Delta^2 S_1 = \Delta^2 \sum\limits_{i=0}^{\frac{t_B}{\Delta} - 2} \left( \frac{t_A}{\Delta} + i \right) &= \left( t_B - \Delta \right) t_A + \Delta^2 \sum\limits_{i=1}^{\frac{t_B}{\Delta} - 2} i \\
&= \left( t_B - \Delta \right) t_A + \frac{1}{2}\left( t_B - \Delta \right) \left( t_B - 2\Delta \right).
\end{align*}
and the second sum to
\begin{align*}
\Delta^2 S_2 = \Delta^2 \sum\limits_{j=1}^{\frac{t_A}{\Delta} - 1} \left( \frac{t_A}{\Delta} + \frac{t_B}{\Delta} - 1 - j\right) &= \left( t_A - \Delta \right) \left( t_A + t_B - \Delta \right) - \Delta^2 \sum\limits_{j=1}^{\frac{t_A}{\Delta} - 1} j \\
&= \left( t_A - \Delta \right) \left( t_A + t_B - \Delta \right) - \frac{1}{2}t_A\left( t_A - \Delta \right).
\end{align*}
Inserting these results in (\ref{eq:frac}) and considering $\Delta \longrightarrow 0$ yields
\begin{align*}
P^* &=  \dfrac{ t_At_B + \frac{1}{2}t_B^2 +(t_A + t_B)( T - t_A - t_B) + t_A(t_A + t_B) - \frac{1}{2}t_A^2 }{T^2} \\[0.8em]
 &= \dfrac{ t_At_B + \frac{1}{2}t_B^2 +(t_A + t_B)( T - t_B) - \frac{1}{2}t_A^2 }{T^2} \\[0.8em]
 &=  \dfrac{ t_AT + t_BT - \frac{1}{2}t_A^2 - \frac{1}{2}t_B^2 }{T^2} =  \dfrac{ t_A + t_B }{T} - \dfrac{ t_A^2 + t_B^2 }{2T^2},
\end{align*}
which concludes the proof.
\end{proof}

\begin{center}
--- Examples ---
\end{center}

\noindent \textit{Example I (modified)} --- John looks outside once for 5 minutes and a blue car occurs once for 2 minutes during 1 hour. He will see a blue car with a probability of about $11.26 ~\%$:
\begin{equation*}
P^*(60,5,2) = \dfrac{5+2}{60} - \dfrac{5^2+2^2}{2\cdot 60^2} \approx 0.1126~.
\end{equation*}

\noindent \textit{Example II} --- Two events with $t_A = t_B = 1~sec$ will overlap during a total time $T = 3~sec$ with a probability of about $55.56 ~\%$:
\begin{equation*}
P^*(3,1,1) = \dfrac{1+1}{3} - \dfrac{1^2+1^2}{2\cdot 3^2} \approx 0.5556~.
\end{equation*}

\section{Approximation of Universal Equation}\label{sec:app}

Before we derive a more complex universal equation (see Section \ref{sec:unieq}), we will find an approximation, which shall apply for any given $n_A$ and $n_B$.

\begin{theorem}
The probability of $A$ and $B$ to overlap at some point during the total time $T$, with $A$ and $B$ occurring $n_A$ and $n_B$ times for $t_A$ and $t_B$, can be approximated by
\begin{equation*}
P(T,t_A,t_B,n_A,n_B) \approx 1 -  \left( 1 - n_A  \dfrac{t_A + t_B}{T} \right)^{n_B}.
\end{equation*}
\end{theorem}

\begin{proof}[Proof of Theorem]

Since $P = 1 - \bar{P}$, we will establish an equation for the probability $\bar{P}$ for the events not to occur at the same time at any point. Let $k \in \{ 0 , 1 , ... , t_B-1 \}$ number the $n_B$ occurrences of $B$. There are
\begin{equation}\label{eq:k}
\frac{T}{\Delta} - k \cdot \frac{t_B}{\Delta}
\end{equation}
possible positions for the $k$-th $B$ occurrence to start, but we consider $T \gg t_B$ and approximate (\ref{eq:k}) with $T / \Delta$. Every $B$ occurrence may overlap at about
\begin{equation*}
\frac{t_A}{\Delta} + \frac{t_B}{\Delta} 
\end{equation*}
positions (see (\ref{eq:possPositions})) with any occurrence of $A$. Thus we have for the probability $\bar{p}$ of a $B$ occurrence \emph{not} to overlap at some point with any $A$ occurrence
\begin{equation*}
\bar{p} \approx 1 - \dfrac{n_A \cdot \left( \frac{t_A}{\Delta} + \frac{t_B}{\Delta} \right)}{ T / \Delta } = 1 - n_A  \dfrac{t_A + t_B}{T}
\end{equation*}
To find $\bar{P}$ this has to be true for every occurrence of $B$ and we get
\begin{equation*}
P = 1 - \bar{P} = 1 - \left( \bar{p} \right)^{n_B} \approx 1 - \left( 1 - n_A \dfrac{t_A + t_B}{T} \right)^{n_B},
\end{equation*}
which confirms the approximation.
\end{proof}

\begin{center}
--- Example ---
\end{center}

\noindent \textit{Example I} --- John will see a blue car with a probability of about $83.85 ~\%$:
\begin{equation*}
P(120,3,1,5,10) \approx 1 -  \left( 1 - 5 \cdot \dfrac{3 + 1}{120} \right)^{10} \approx 0.8385 ~.
\end{equation*}
In Example I in Section \ref{sec:unieq} the error of this result will be shown to be insignificant.

\section{Universal Equation}
\label{sec:unieq}

\begin{theorem}\label{theo:uniq}
The probability of $A$ and $B$ to overlap at some point during the total time $T$, with $A$ and $B$ occurring $n_A$ and $n_B$ times for $t_A$ and $t_B$, is given by
\begin{equation*}
P(T,t_A,t_B,n_A,n_B) = 1 - \dfrac{ \left( T^+ - t_A n_A - t_B n_B \right)^{n_A + n_B} }{ \left( T^+ - t_A n_A  \right)^{n_A} \left( T^+ - t_B n_B \right)^{n_B} } \pm E^\pm
\end{equation*}
with $T^+ := T + \frac{t_A + t_B}{2}$ and a small error term
\begin{equation*}
0 \leq E^\pm \leq \dfrac{ \left( T^+ - t_A n_A - t_B n_B \right)^{n_A + n_B} }{ \left( T^+ - t_A n_A  \right)^{n_A} \left( T^+ - t_B n_B \right)^{n_B} } \cdot \left(  \left(  \dfrac{\alpha_A(\alpha_B + \tau ) }{ \alpha_B(\alpha_A + \tau) } \right)^{n_A+n_B}  - 1 \right)
\end{equation*}
with $\alpha_A := T + t_A - t_A n_A - t_B n_B$, $\alpha_B := \alpha_A - \frac{t_A-t_B}{2}$ and $\tau := \max \{ t_An_A , t_Bn_B \}$.
\end{theorem}

\begin{lemma}\label{lem:stirling} For constant $y \in \mathbb{N}$ we have
\begin{equation*}
\lim_{X \longrightarrow \infty} \dfrac{(X+y)!}{X! X^y} = 1.
\end{equation*}
\end{lemma}
\begin{proof}
The above fraction can be rearranged to
\begin{equation*}
\lim_{X \longrightarrow \infty} \dfrac{(X+1) \cdot ... \cdot (X+y)}{X^y} = \lim_{X \longrightarrow \infty} \prod^y_{l=1} \left( 1 +\dfrac{l}{X} \right) = 1.
\end{equation*}
\end{proof}

\begin{lemma}\label{lem:bruch} For constant $u,v \in \mathbb{R}$ with $u \leq v$ following inequality holds true:
\begin{equation*}
\dfrac{\alpha_B + u}{\alpha_A + u} \leq \dfrac{\alpha_B + v}{\alpha_A + v}.
\end{equation*}
\end{lemma}
\begin{proof}
Rearranging the inequality yields
\begin{align}
\nonumber (\alpha_A + v) (\alpha_B + u)  &\leq  (\alpha_A + u) (\alpha_B + v) \\
\nonumber \alpha_Au + \alpha_Bv &\leq \alpha_Av + \alpha_Bu \\
\alpha_B (v - u)  &\leq  \alpha_A (v - u). \label{eq:ineq}
\end{align}
Inequality (\ref{eq:ineq}) holds true since $v-u \geq 0$ and $\alpha_B := \alpha_A - \frac{1}{2}t_A + \frac{1}{2}t_B \leq \alpha_A$ due to Definition (a) in Section \ref{sec:intro} that $t_A \geq t_B$.
\end{proof}

\begin{proof}[Proof of Theorem]
Similar to the approximation we determine the probability $\bar{P}$ for \emph{no} overlap to occur. Therefore we divide the time in parts of size $\Delta$ again and count the number of possible arrangements of all $A$ and $B$ occurrences. There are
\begin{equation*}
\frac{T}{\Delta} - \left(\frac{t_A}{\Delta} - 1\right) \cdot n_A
\end{equation*} 
positions for the $n_A$ occurrences of $A$ left, so that they do not overlap with themselves. Thus the number of possible arrangements of the $A$ occurrences is given by
\begin{equation}\label{eq:B}
\binom{\frac{T}{\Delta} - \left(\frac{t_A}{\Delta} - 1\right) \cdot n_A}{n_A}
\end{equation}
and for the B occurrences by
\begin{equation}\label{eq:BB}
\binom{\frac{T}{\Delta} - \left(\frac{t_B}{\Delta} - 1\right) \cdot n_B}{n_B}.
\end{equation}
The number of ways to order the $n_A$ occurrences of $A$ and the $n_B$ occurrences of $B$ is
\begin{equation}\label{eq:q}
\binom{n_A + n_B}{n_A}.
\end{equation}
Similar to (\ref{eq:B}) and (\ref{eq:BB}) we can express the number of ways to arrange the $A$ and $B$ occurrences per order of (\ref{eq:q}) as
\begin{equation}\label{eq:occorder}
\binom{\frac{T}{\Delta} - \left(\frac{t_A}{\Delta} - 1\right) \cdot n_A - \left(\frac{t_B}{\Delta} - 1\right) \cdot n_B}{n_A + n_B}.
\end{equation}
Since the probability $\bar{P}$ describes the ratio of arrangements without overlap\footnote{number of arrangements without overlap (\ref{eq:occorder}) per orders (\ref{eq:q})} to the total number of possible arrangements of the $A$ and $B$ occurrences, we have\footnote{see notation in Section \ref{sec:not}}
\begin{equation*}
\bar{P} = \lim_{\Delta \longrightarrow 0} \dfrac{\binom{n_A + n_B}{n_A} \binom{\frac{T}{\Delta} - \left(\frac{t_A}{\Delta} - 1\right) \cdot n_A - \left(\frac{t_B}{\Delta} - 1\right) \cdot n_B}{n_A + n_B}}{\binom{\frac{T}{\Delta} - \left(\frac{t_A}{\Delta} - 1\right) \cdot n_A}{n_A} \binom{\frac{T}{\Delta} - \left(\frac{t_B}{\Delta} - 1\right) \cdot n_B}{n_B}} = \lim_{\Delta \longrightarrow 0} \dfrac{\binom{n_A + n_B}{n_A} \binom{T' - N_A - N_B}{n_A + n_B} }{\binom{T' - N_A}{n_A} \binom{T' - N_B }{n_B}}.
\end{equation*}
Rewriting the binomial coefficients and rearrange the fraction yields
\begin{align}
\nonumber \bar{P} &= \lim_{\Delta \longrightarrow 0} \dfrac{ (n_A + n_B)!  (T' - N_A - N_B)! n_A! (T' - N_A - n_A)! n_B! (T' - N_B - n_B)! }{(T' - N_A)!(T' - N_B)!n_A!n_B!(n_A+n_B)!(T' - N_A - N_B - n_A - n_B)! } \\[0.8em]
\nonumber &= \lim_{\Delta \longrightarrow 0} \dfrac{ (T' - N_A - N_B)! (T' - N_A - n_A)! (T' - N_B - n_B)! }{(T' - N_A)! (T' - N_B)! (T' - N_A - N_B - n_A - n_B)! } \\[0.8em]
&= \lim_{\Delta \longrightarrow 0} \dfrac{ (M + n_A + n_B)! M_A! M_B! }{ (M_A + n_A)! (M_B + n_B)! M! }. \label{eq:MMM}
\end{align}
Due to the limit we have $\Delta \longrightarrow 0$ and thus $M_A, M_B, M \longrightarrow \infty$. Therefore we are allowed to apply Lemma \ref{lem:stirling} on (\ref{eq:MMM}):
\begin{align*}
\bar{P} &= \lim_{\Delta \longrightarrow 0}  \dfrac{ M! M^{n_A + n_B} M_A! M_B! }{ M_A! M_A^{n_A} M_B!  M_B^{n_B} M! } = \lim_{\Delta \longrightarrow 0} \dfrac{ M^{n_A + n_B} }{ M_A^{n_A} M_B^{n_B} } \\[0.8em]
&= \lim_{\Delta \longrightarrow 0} \dfrac{ \left( T' - N_A - n_A - N_B - n_B \right)^{n_A + n_B} }{ \left(T' - N_A - n_A \right)^{n_A} \left(T' - N_B - n_B \right)^{n_B} } \\[0.8em]
&= \lim_{\Delta \longrightarrow 0} \dfrac{ \left( \frac{T}{\Delta} - \frac{t_A}{\Delta} n_A - \frac{t_B}{\Delta} n_B \right)^{n_A + n_B} }{ \left(\frac{T}{\Delta} - \frac{t_A}{\Delta} n_A  \right)^{n_A} \left(\frac{T}{\Delta} - \frac{t_B}{\Delta} n_B \right)^{n_B} } \\[0.8em] 
&= \dfrac{ \left( T - t_A n_A - t_B n_B \right)^{n_A + n_B} }{ \left( T - t_A n_A  \right)^{n_A} \left( T - t_B n_B \right)^{n_B} }.
\end{align*}
This equation yields the probability in case that the occurrences may not overlap the total time. Since we defined in (c) in Section \ref{sec:intro} the occurrences may overlap with the end of the total time, we have to extend the total time $T$ by some $\delta > 0$:
\begin{equation*}
\bar{P}_\delta := \dfrac{ \left( T + \delta - t_A n_A - t_B n_B \right)^{n_A + n_B} }{ \left( T + \delta - t_A n_A  \right)^{n_A} \left( T + \delta - t_B n_B \right)^{n_B} }
\end{equation*}
Obviously, $t_B \leq \delta \leq t_A$ and we choose $\delta^* := \frac{t_A + t_B}{2}$. Therefore let us redefine
\begin{equation*}
\bar{P} := \bar{P}_{\delta^*} \pm E^\pm = \dfrac{ \left( T^+ - t_A n_A - t_B n_B \right)^{n_A + n_B} }{ \left( T^+ - t_A n_A  \right)^{n_A} \left( T^+ - t_B n_B \right)^{n_B} } \pm E^\pm
\end{equation*}
with $0 \leq E^\pm$ as small error term. $E^\pm$ approaches $0$ if $\vert t_A - t_B\vert$ approaches $0$ or if the ratio between $T$ and the event time $t_An_A + t_Bn_B$ increases\footnote{Notation: '$\sim$' denotes 'proportional to'}
\begin{equation*}
E^\pm \sim \vert t_A - t_B \vert \cdot \dfrac{ t_An_A + t_Bn_B }{T}.
\end{equation*}
In order to determine the maximum error we consider, that the probability for an overlap decreases with longer $T$. Since we defined $t_A \geq t_B$ (see Definition (a) in Section \ref{sec:intro}), we have $P_{\delta^*} - P_{t_A} \geq P_{t_B} - P_{\delta^*}$ and the maximum error is given by
\begin{align}
\nonumber E^\pm &\leq P_{\delta^*} - P_{t_A} = (1- \bar{P}_{\delta^*}) - (1 -\bar{P}_{t_A} ) = \bar{P}_{\delta^*} \cdot \left( \dfrac{\bar{P}_{t_A}}{\bar{P}_{\delta^*}} - 1 \right) \\
\nonumber &= \bar{P}_{\delta^*} \cdot \left( \left( \dfrac{\alpha_A }{ \alpha_B} \right)^{n_A+n_B} \cdot \left( \dfrac{\alpha_B + t_Bn_B }{ \alpha_A + t_Bn_B } \right)^{n_A} \cdot \left( \dfrac{\alpha_B + t_An_A }{ \alpha_A + t_An_A } \right)^{n_B} - 1 \right) \\ 
\nonumber &\leq \bar{P}_{\delta^*} \cdot \left( \left( \dfrac{\alpha_A }{ \alpha_B} \right)^{n_A+n_B} \cdot \left( \dfrac{\alpha_B + \tau }{ \alpha_A + \tau} \right)^{n_A+n_B} - 1 \right) \tag{see Lemma \ref{lem:bruch}} \\
\nonumber &= \bar{P}_{\delta^*} \cdot \left(  \left(  \dfrac{\alpha_A(\alpha_B + \tau ) }{ \alpha_B(\alpha_A + \tau ) } \right)^{n_A+n_B}  - 1 \right).
\end{align}
Finally the probability for an overlap is given by $P = 1 - \bar{P}$, which concludes the proof.
\end{proof}

\begin{center}
--- Example ---
\end{center}

\noindent \textit{Example I} --- John will see a blue car with a probability of about $85.46 \pm 1.77 ~\%$. Since $T^+ = 120 + \frac{3+1}{2}  = 122$, $\alpha_A = 120 + 3 - 3 \cdot 5 - 1 \cdot 10 = 98$ and $\alpha_B = 98 - \frac{3 - 1}{2}  = 97$, we have
\begin{equation*}
P(120,3,1,5,10) = 1 - \dfrac{ \left( 122 - 3 \cdot 5 - 1 \cdot 10 \right)^{5 + 10} }{ \left( 122 - 3 \cdot 5  \right)^{5} \left( 122 - 1 \cdot 10 \right)^{10} } \pm E^\pm  \approx 0.8546 \pm E^\pm
\end{equation*}
and, due to $\max\{ 3 \cdot 5 , 1 \cdot 10 \} = 3 \cdot 5$,
\begin{equation*}
E^\pm \leq 85.46 ~\%  \cdot  \left(  \left(  \dfrac{98 \cdot (97 + 3 \cdot 5) }{ 97 \cdot (98 + 3 \cdot 5) } \right)^{5 + 10}  - 1 \right) \approx 1.77~ \% ~.
\end{equation*}
This confirms the precision of the approximation in Section \ref{sec:app}. \\

\noindent \textit{Example II} --- Two events with $t_A = t_B = 1~sec$ will overlap during a total time $T = 3~sec$ at some point with a probability of about $55.56 ~\%$. Since $T^+ = 3 + \frac{1+1}{2} = 4$, we have
\begin{equation*}
P(3,1,1,1,1) = 1 - \dfrac{ \left( 4 - 1 \cdot 1 - 1 \cdot 1 \right)^{1+1} }{ \left( 4 - 1 \cdot 1  \right) \left( 4 - 1 \cdot 1 \right) } \pm 0  \approx 0.5555 \pm 0~
\end{equation*}
without any error due to $\vert t_A - t_B \vert = \vert 1 - 1 \vert = 0$. This result is consistent with Section \ref{sec:preciseEqn}.

\section{ Remark: $n_A$ and $n_B$ as random parameters}\label{sec:nn}

Instead of determining $n_A$ and $n_B$ as precise numbers of occurrences, it is more natural to define $\rho_A$ and $\rho_B$, which represent the probability $p$ for the event to occur in a given time $s$: $\rho := p / s$. In order to calculate the probability for a possible overlap, the number of expected occurrences $n = \rho \cdot T$ has to be calculated first. The universal equation would simply change to
\begin{equation*}
P(T,t_A,t_B,\rho_A,\rho_B) = 1 - \dfrac{ \left( T^+ - t_A T \rho_A - t_B T \rho_B \right)^{T \rho_A + T \rho_B} }{ \left( T^+ - t_A T \rho_A  \right)^{T \rho_A} \left( T^+ - t_B T \rho_B \right)^{T \rho_B} } \pm E^\pm.
\end{equation*}

\textbf{} \\  \noindent\textbf{Acknowledgement} --- 
The author would like to thank T. Andrews for his advice and helpful comments as well as Alice de Saupaio Kalkuhl for reviewing this paper. 

\nocite{ma}
\bibliography{Paper} 
\bibliographystyle{plain}

\end{document}